\documentclass[12pt,reqno,openbib,runningheads,a4paper]{amsart}%
\usepackage{eurosym}
\usepackage{graphicx}
\usepackage{amssymb}
\usepackage{amsfonts}
\usepackage{amsmath}
\usepackage{graphicx}
\usepackage{epstopdf}
\usepackage[authoryear]{natbib}
\usepackage[legalpaper,bookmarks=true,colorlinks=true,linkcolor=blue,citecolor=blue]{hyperref}

\providecommand{\U}[1]{\protect\rule{.1in}{.1in}}
\providecommand{\U}[1]{\protect\rule{.1in}{.1in}}
\newtheorem{theorem}{Theorem}[section]

\newtheorem{proposition}{Proposition}[section]

\makeatletter
\renewcommand{\@biblabel}[1]{}
\makeatother
\begin{document}
	
	\begin{center}
		{\Large \textbf{Robust estimator of distortion risk premiums for
				heavy-tailed losses}}\medskip \medskip
		
		Brahim Brahimi{\footnote{{\texttt{brah.brahim@gmail.com}}} and }Zoubir
		Kenioua\medskip
	\end{center}
	
	{\small \textit{Laboratory of Applied Mathematics, Mohamed Khider
			University, Biskra, Algeria}} \medskip
	
	\noindent\textbf{Abstract}\medskip
	
	\noindent We use the so-called t-Hill tail index estimator proposed by \cite%
	{Fab01}, rather than Hill's one, to derive a robust estimator for the
	distortion risk premium of loss. Under the second-order condition of regular
	variation, we establish its asymptotic normality. By simulation study, we
	show that this new estimator is more robust than of \cite{NeMe09} both for
	small and large samples.\medskip
	
	\medskip
	
	
	\section{\textbf{Introduction\label{sec1}}}
	
	\noindent In many important applications in Finance, Actuarial Science,
	Hydrology, Insurance, one of most crucial topics is the determination of the
	amounts of losses of a heavy-tailed risks. In literature there are many
	possible definitions of risk according to the investment objectives, so in
	some sense risk itself is a subjective concept as well as the risk measure for an investor. Thus, the existence of a unique risk that solves the investor's
	problems is not confirmed. However, one must distinguish between risks. The
	concept of coherence due to the paper of \citep{ArDeEbHe99} which categorize
	risks by good or bad risks. For this reason and to improve the performance
	of investor's strategies we may identify those risk measures and the
	appearance of the risk himself heavy tails, asymmetries,... Most
	of this risk measures, used are special cases of Wang's distortion premium %
	\citep[][]{Wang96}, defined as follows 
	\begin{equation}
	\Pi \lbrack \psi ;F]=\int_{0}^{\infty }\psi (1-F(x))dx,  \label{eq1}
	\end{equation}%
	where $\psi :[0,1]\rightarrow \lbrack 0,1]$ is a non-decreasing function
	called distortion function, such that $\psi (0)=0$ and $\psi (1)=1.$ The
	distortion functions $\psi $ are concave, which makes the corresponding
	distortion premiums $\Pi \lbrack \psi ;F]$ coherent \citep{ArDeEbHe99} as
	proved by \cite{WiHa99}. In this paper, we suppose that the distortion
	functions $t\longmapsto \psi \left( t^{-1}\right) $ is regularly varying at
	infinity with index of regular variation $\rho \geq 1,$ such that%
	\begin{equation}
	\psi \left( t^{-1}\right) =t^{-1/\rho }\mathcal{L}_{\psi }\left( t\right) ,
	\label{g}
	\end{equation}%
	where $t\longmapsto \mathcal{L}_{\psi }\left( t\right) $ is slowly varying
	as infinity, that is $\mathcal{L}_{\psi }\left( tx\right) /\mathcal{L}_{\psi
	}\left( t\right) \rightarrow 1$ as $t\rightarrow \infty ,$ for any $x>0.$
	Note that a non negative random variable (rv) $X$ with finite mean and a
	cumulative distribution function (cdf) $F$ is called heavy-tailed, if $1-F$
	is regularly varying with index $-1/\gamma <0$ ( notation: $F\in \mathcal{RV}%
	_{\left( -1/\gamma \right) })$, that is%
	\begin{equation}
	\lim_{t\rightarrow \infty }\frac{1-F\left( tx\right) }{1-F\left( t\right) }%
	=x^{-1/\gamma },\text{ for }x>0.  \label{first-condition}
	\end{equation}%
	In particular, the proportional-hazards premium (see, \citealp{Roletal99}, page
	82). 
	\begin{equation}
	\Pi \lbrack \rho ;F]=\int_{0}^{\infty }\left( 1-F(x)\right) ^{1/\rho }dx,
	\label{Psi-ro}
	\end{equation}%
	with the concave distortion function $\psi (t)=t^{1/\rho }$ for every $\rho
	\geq 1.$ Since we are concerned with with heavy-tailed losses with infinite
	second moment, then by following \cite{BrMeNeRi}, we assume that $\gamma \in
	(1/2,1)$ and $\rho \gamma \in \left( 0,1\right) ,$ thus we will work with 
	\begin{equation}
	1/2<\gamma <1/\rho .  \label{condition gamabeta}
	\end{equation}%
	Suppose that we have an independent and identically distributed (iid) sample 
	$X_{1},...,X_{n}$ of rv $X$ of size $n$ with a cdf $F$ satisfying condition (%
	\ref{first-condition}) and let denote by $X_{1:n}\leq ...\leq X_{n:n}$ the
	corresponding order statistics. Also, let $1<k=k_{n}$ be the number of
	extreme observations used in the computation of the tail index. We assume
	that $k$ satisfies the conditions 
	\begin{equation}
	1<k<n,\text{ }k\rightarrow \infty \text{ and }k/n\rightarrow 0\text{ as }%
	n\rightarrow \infty .  \label{K}
	\end{equation}%
	\cite{NeMe09} proposed an alternative estimator of (\ref{Psi-ro}) and
	establish its asymptotic normality by using the Weissman's estimator of the
	high quantile $q_{t}=F^{\leftarrow }\left( 1-t\right) $ defined by 
	\begin{equation*}
	\widehat{q}_{t}=\left( k/n\right) ^{\widehat{\gamma }^{H}}X_{n-k:n}t^{-%
		\widehat{\gamma }^{H}},\text{ }t\downarrow 0,
	\end{equation*}%
	where $F^{\leftarrow }$\ denotes the generalized inverse of $F$ and%
	\begin{equation}
	\widehat{\gamma }^{H}=\widehat{\gamma }^{H}\left( k\right) :=\frac{1}{k}%
	\sum\limits_{i=1}^{k}\log X_{n-i+1:n}-\log X_{n-k:n},  \label{Hill}
	\end{equation}%
	is the well-known Hill estimator \citep[][]{Hill75} of the tail index $%
	\gamma .$ For a fixed aversion parameter $\rho ,$ their estimator is given by%
	\begin{equation}
	\widehat{\Pi }_{n}\left( \widehat{\gamma }^{H},k\right) :=\left( k/n\right)
	^{1/\rho }\frac{X_{n-k:n}}{1-\widehat{\gamma }^{H}\rho }+\sum_{i=k+1}^{n}%
	\left( \left( i/n\right) ^{1/\rho }-\left( \left( i-1\right) /n\right)
	^{1/\rho }\right) X_{n-i+1:n}.  \label{NeMe}
	\end{equation}%
	The Hill estimator is a pseudo-maximum likelihood estimator based on the
	exponential approximation of normalized log-spacings $Y_{j}=j\left( \log
	X_{j:n}-\log X_{j+1:n}\right) $ for $j=1,...,k.$ In practice, the Hill
	estimator depends on the choice of the sample fraction $k$ and is inherently
	not very robust to large values $Y_{j},$ which makes the estimator proposed
	by \cite{NeMe09} sensitive. This constitutes a serious problem in terms of
	bias and root mean squared error (RMSE). To improve the quality of $\widehat{%
		\Pi }_{n}\left( \widehat{\gamma }^{H},k\right) $, instead of Hill's one, we
	propose to estimate the tail index $\gamma $ by the so-called t-Hill
	estimator, proposed by \cite{Fab01}, given by its harmonic mean%
	\begin{equation}
	\widehat{\gamma }=\widehat{\gamma }\left( k\right) :=\left( \frac{1}{k}%
	\sum_{j=1}^{k}\frac{X_{n-k:n}}{X_{n-j+1:n}}\right) ^{-1}-1  \label{t-Hill}
	\end{equation}%
	known as score moment estimation (t-score or t-estimation method). The
	latter is more robust than the classical Hill estimator $\widehat{\gamma }%
	^{H}$ defined in (\ref{Hill}) (see \citealp{StFaS12} and the asymptotic
	normality is given in Theorem 2 of \citealp{BeScSt14}). For other robust
	estimators for $\gamma $ we referred to \cite{PeWe01}, \cite{JuSc04}, \cite%
	{VaBeChHu07} and \cite{KiLe08}. The rest of the paper is organized as
	follows, in Section \ref{sec2} we present a construction of a robust
	estimator of $\Pi \lbrack \psi ;F]$ in the case of heavy-tailed losses. In
	Section \ref{sec3} we establish its asymptotic normality. In Section \ref%
	{sec4} we carry out a simulation study to illustrate empirical performance
	and robustness of the estimator. Concluding notes are given in Section \ref%
	{sec5}. Proofs are gathered in Section \ref{sec6}.\medskip 
	
	\noindent Throughout the paper, we use the standard notation $\overset{P}{%
		\rightarrow }$ for the convergence in probability and $\mathcal{N}\left( \mu
	,\sigma \right) $ to denote a normal rv with mean $\mu $ and variance $%
	\sigma .$
	
	\section{\textbf{Defining the estimator\label{sec2}}}
	
	\noindent By using the generalized inverse $F^{\leftarrow }$ and for a fixed
	distortion function $\psi ,$ we may rewrite (\ref{eq1}) into 
	\begin{equation}
	\Pi _{\psi }[X]:=-\int_{0}^{1}\psi \left( s\right) dF^{\leftarrow }\left(
	1-s\right) .  \label{DRMQ}
	\end{equation}%
	The empirical estimator of the risk premium $\Pi _{\psi }[X]$ is obtained by
	substituting $F^{\leftarrow }$ on the right-hand side of equation (\ref{DRMQ}%
	) by its empirical counterpart $F_{n}^{\leftarrow }{\left( s\right) :=\inf
		\{x\in \mathbb{R}:}$ ${F}_{n}{\left( x\right) \geq s\},}$ $0<s\leq 1,$
	associated to the empirical cdf defined on the real line, defined by $%
	{\normalsize F_{n}\left( x\right) :=n}^{-1}\#\left\{ X_{i}\leq x,1\leq i\leq
	n\right\} $ where $\#A$ denote the cardinality of a set $A.$\ After
	straightforward computations, we obtain the formula 
	\begin{equation*}
	\Pi _{n}[X]:=\int_{0}^{1}F_{n}^{\leftarrow }\left( 1-s\right) d\psi \left(
	s\right)
	\end{equation*}%
	which may be rewritten, in terms of $X_{1:n},...,X_{n:n},$ as an $L$%
	-statistic 
	\begin{equation}
	\Pi _{n}[X]=\sum_{i=1}^{n}c_{i,n}\left( \psi \right) X_{n-i+1:n},
	\label{psin}
	\end{equation}%
	where 
	\begin{equation}
	c_{i,n}\left( \psi \right) \equiv \psi \left( i/n\right) -\psi \left( \left(
	i-1\right) /n\right) .  \label{ai}
	\end{equation}%
	The form (\ref{psin}) is a linear combinations of the order statistics %
	\citep[see,][page 260]{SW86}. The limit behavior was discussed by many
	authors: \cite{Chall67}, \cite{Stigler74}, \cite{Mason81}, \cite{JoZi03}
	(see its Theorem 3.2 in the case that $X$ is not heavy-tailed) and in \cite%
	{BrMeNeRi} (in heavy-tailed case).
	
	\subsection{\textbf{Heavy-tailed losses case}}
	
	\noindent Let $X$\ be a non-negative rv with cdf $F\in \mathcal{RV}_{\left(
		-1/\gamma \right) }.$ The condition (\ref{first-condition}) is equivalent to%
	\begin{equation}
	\lim_{t\rightarrow 0}\frac{F^{\leftarrow }\left( 1-tx\right) }{F^{\leftarrow
		}\left( 1-t\right) }=x^{-\gamma },\text{ for every }x>0.
	\label{First-Cond-Qua}
	\end{equation}%
	We say that the function $s\rightarrow F^{\leftarrow }\left( 1-s\right) $
	satisfying condition (\ref{First-Cond-Qua}) is regularly varying at zero
	with the index $(-\gamma )<0.$ The parameter $\gamma $ is called the tail
	index or extreme value index (EVI). A various tail index estimators have
	been suggested in the literature, based for instance of the conventional
	maximum likelihood method, moment estimation, ... (see, e.g. \citealp{Hill75}%
	, \citealp{Pick75}, \citealp{Dekeretal89}. \citealp{Csorgoetal85} and %
	\citealp{Drees95}). For the robustness and bias reduction (see, i.e. %
	\citealp{PenQui04} and \citealp{StFaS12}). The regular-variation condition
	itself is not sufficient for establishing asymptotic distributions. To this
	end, we suppose that cdf $F$ satisfy the well-known by the second-order
	condition of regular variation with second-order parameter $\tau \leq 0$,
	that is: there exists a function $t\rightarrow a(t)$ with constant sign at
	infinity and converges to $0$ as $t\rightarrow \infty $ such that 
	\begin{equation}
	\underset{t\rightarrow \infty }{\lim }\dfrac{\overline{F}\left( tx\right) /%
		\overline{F}\left( t\right) -x^{-1/\gamma }}{a\left( t\right) }=x^{-1/\gamma
	}\dfrac{x^{\tau /\gamma }-1}{\gamma \tau },  \label{second-order}
	\end{equation}%
	for every $x>0.$ When $\tau =0,$ then the ratio $\dfrac{x^{\tau /\gamma }-1}{%
		\gamma \tau }$ should be interpreted as $\log x.$ In terms of the quantile
	function $F^{\leftarrow },$ condition (\ref{second-order}) is equivalent to
	the following one 
	\begin{equation}
	\lim_{t\rightarrow 0}\frac{\dfrac{F^{\leftarrow }\left( 1-tx\right) }{%
			F^{\leftarrow }\left( 1-t\right) }-x^{-\gamma }}{A\left( t\right) }%
	=x^{-\gamma }\frac{x^{\tau }-1}{\tau },  \label{second-cond}
	\end{equation}%
	for every $x>0,$ where $A\left( t\right) :=\gamma ^{2}a\left( F^{\leftarrow
	}\left( 1-t\right) \right) ,$ (see \citeauthor{deHS96}, \citeyear{deHS96} or
	Theorem 3.2.9 in \citeauthor{deHF06}, \citeyear[page 48]{deHF06}). The
	Weissman estimator \citep[][]{Weis78} of high quantiles $F^{\leftarrow }$ is
	given by 
	\begin{equation}
	F_{n}^{\leftarrow \left( W\right) }(1-s):=(k/n)^{\widehat{\gamma }%
	}X_{n-k:n}s^{-\widehat{\gamma }},\text{ }s\downarrow 0.  \label{wies}
	\end{equation}%
	\medskip The formula (\ref{DRMQ}) can be split into%
	\begin{equation}
	\Pi _{\psi }[X]=-\int_{0}^{k/n}\psi \left( s\right) dF^{\leftarrow
	}(1-s)-\int_{k/n}^{1}\psi \left( s\right) dF^{\leftarrow }(1-s).  \label{pi}
	\end{equation}%
	By using an integration by part to the second integral yields%
	\begin{eqnarray*}
		\Pi _{\psi } &=&\psi \left( k/n\right) F^{\leftarrow
		}(1-k/n)-\int_{0}^{k/n}\psi \left( s\right) dF^{\leftarrow
	}(1-s)+\int_{k/n}^{1}F^{\leftarrow }(1-s)d\psi \left( s\right)  \\
	&:&=\Pi _{\psi }^{\left( 1\right) }+\Pi _{\psi }^{\left( 2\right) }+\Pi
	_{\psi }^{\left( 3\right) }.
\end{eqnarray*}%
A simple estimator of $\Pi _{\psi }^{\left( 1\right) }$ is 
\begin{equation}
\Pi _{\psi ,n}^{\left( 1\right) }:=\psi \left( k/n\right) X_{n-k,n}.
\label{pi1}
\end{equation}%
To estimate $\Pi _{\psi }^{\left( 2\right) }$, we note that $\widehat{\gamma 
}$ is a consistent estimator for $\gamma $ \citep[][]{StFaS12} and since $%
\rho <1/\gamma $, by substituting $F_{n}^{\leftarrow \left( W\right) }(1-s)$
given in (\ref{wies}) instead of $F^{\leftarrow }(1-s)$ and integrating
yield the following estimator%
\begin{equation}
\Pi _{\psi ,n}^{\left( 2\right) }:=\widehat{\gamma }\left( k/n\right) ^{%
	\widehat{\gamma }}X_{n-k:n}\int_{0}^{k/n}s^{-\widehat{\gamma }-1}\psi \left(
s\right) ds.  \label{pi2}
\end{equation}%
Finally, by plugging $F_{n}^{\leftarrow }$ instead of $F^{\leftarrow }$ on
second integral of Equation (\ref{pi}) we obtain the estimator 
\begin{equation}
\Pi _{\psi ,n}^{\left( 3\right) }:=\sum_{i=k+1}^{n}c_{i,n}\left( \psi
\right) X_{n-i+1:n},  \label{pi3}
\end{equation}%
of $\Pi _{\psi }^{\left( 3\right) },$ where $F_{n}^{\leftarrow }{\normalsize %
	\left( s\right) :=\inf \left\{ x\in \mathbb{R}:F_{n}\left( x\right) \geq
	s\right\} ,\;0<s\leq 1,}$ denote the sample quantile function associated to
the empirical cdf defined on the real line by ${\normalsize F_{n}\left(
	x\right) :=n}^{-1}\sum\nolimits_{i=1}^{n}\mathbb{I}\left( X_{i}\leq x\right)
,${\normalsize \ }with\ $\mathbb{I}\left( \cdot \right) $\ being\ the\
indicator\ function and the coefficients $c_{i,n}\left( \psi \right) $ are
given in (\ref{ai}). The final form of our estimator%
\begin{equation*}
\widetilde{\Pi }_{\psi ,n}:=X_{n-k,n}\left( \psi \left( k/n\right) +\widehat{%
	\gamma }\left( k/n\right) ^{\widehat{\gamma }}\int_{0}^{k/n}s^{-\widehat{%
		\gamma }-1}\psi \left( s\right) ds\right) +\sum_{i=k+1}^{n}c_{i,n}\left(
\psi \right) X_{n-i+1:n}.
\end{equation*}%
To establish the asymptotic normality of our estimator and compared with the
estimator proposed by \citeauthor{NeMe09} given in Equation (\ref{NeMe}), we
use the same function $\psi \left( t\right) =t^{1/\rho },$ for $\rho \geq 1$
used in \cite{NeMe09}. In this case our estimator have the following form%
\begin{equation}
\widetilde{\Pi }_{n}\left( \widehat{\gamma },k\right) :=\left( k/n\right)
^{1/\rho }\frac{X_{n-k,n}}{1-\widehat{\gamma }\rho }+\sum_{i=k+1}^{n}\left(
\left( i/n\right) ^{1/\rho }-\left( \left( i-1\right) /n\right) ^{1/\rho
}\right) X_{n-i+1:n}.  \label{Br}
\end{equation}

\section{\textbf{Asymptotic distribution\label{sec3}}}

\noindent We will begin to expose our results as asymptotic representations
theorems in the lines of \cite{BeScSt14}. For that purpose, we need to
describe the probability theory on which they hold. Indeed, we use the
so-called Hungarian construction of \cite{CsCsHM86}. For this we define $%
\left\{ U_{n}(s),0\leq s\leq 1\right\} $, the uniform empirical distribution
function and we consider the order statistics $U_{1:n}\leq ...\leq U_{n:n}$
pertaining to the independent standard uniform rv's $U_{1},U_{2},...,$ we
introduce the uniform empirical quantile function $\left\{ V_{n}(s),0\leq
s\leq 1\right\} $ based on the $n\geq 1$ first observations of $%
U_{1},U_{2},...,$ on $(0,1)$ such that%
\begin{equation*}
V_{n}\left( s\right) =U_{i,n}\quad \text{for}\ \left( i-1\right) /n<s\leq
i/n,\text{ }i=1,...,n,\text{ and }V_{n}\left( 0\right) =U_{1,n}.
\end{equation*}%
Let $\beta _{n}\left( s\right) =\sqrt{n}\left( s-V_{n}\left( s\right)
\right) ,0\leq s\leq 1$, be the corresponding quantile empirical process,
for $n\geq 1.$ \medskip

\noindent We use the well-known by Gaussian approximation given in \cite%
{CsCsHM86} Corollary 2.1. It says that: on the probability space $\left(
\Omega ,\mathcal{A},\mathbb{P}\right) ,$ there exists a sequence of Brownian
bridges $\left\{ \mathbb{B}_{n}\left( s\right) ;\text{ }0\leq s\leq
1\right\} $ has the representation 
\begin{equation*}
\left\{ \mathbb{B}_{n}\left( s\right) ;\text{ }0\leq s\leq 1\right\} \overset%
{D}{=}\left\{ W_{n}\left( s\right) -sW_{n}\left( 1\right) ;\text{ }0\leq
s\leq 1\right\} ,
\end{equation*}%
where $W_{n}$ is a standard Wiener process such that for every $0\leq \zeta
<1/2,$%
\begin{equation}
\sup_{1/n\leq s\leq 1-1/n}\frac{n^{\zeta }\left\vert \beta _{n}\left(
	s\right) -\mathbb{B}_{n}\left( s\right) \right\vert }{\left( s\left(
	1-s\right) \right) ^{1/2-\zeta }}=O_{\mathbb{P}}\left( n^{-\zeta }\right) .
\label{approxi}
\end{equation}%
Our main result is the following

\begin{theorem}
	\label{Theorem2}Let $F$ be a df satisfying (\ref{second-cond}) with $\gamma
	>1/2$ and suppose that $F^{\leftarrow }\left( \cdot \right) $ is
	continuously differentiable on $[0,1).$ Let $k=k_{n}$ satisfying (\ref{K})
	such that $\sqrt{k}A\left( n/k\right) \rightarrow 0$ as $n\rightarrow \infty
	.$ For any $1\leq \rho <1/\gamma ,$ we have%
	\begin{equation*}
	\frac{\sqrt{n}\left( \widetilde{\Pi }_{\psi ,n}-\Pi _{\psi }\right) }{%
		(k/n)^{1/\rho -1/2}F^{\leftarrow }(1-k/n)}\overset{d}{\rightarrow }\mathcal{N%
	}\left( 0,\sigma ^{2}\left( \gamma ,\rho \right) \right) ,\text{ as }%
	n\rightarrow \infty ,
	\end{equation*}%
	where%
	\begin{eqnarray*}
		\sigma ^{2}\left( \gamma ,\rho \right) &=&\gamma ^{2}+\frac{\gamma ^{2}\rho
			\left( \rho -2\rho \gamma ^{2}+2\gamma \right) }{\left( \gamma \rho
			-1\right) ^{2}}+\frac{2\gamma ^{2}}{\left( \rho +\gamma \rho -1\right)
			\left( \rho +2\gamma \rho -2\right) } \\
		&&+\frac{2\gamma }{2\gamma -1}-\frac{2\gamma \rho \left( \rho \gamma
			^{2}-\rho \gamma +1\right) }{\left( \gamma \rho -1\right) \left( \rho
			+\gamma \rho -1\right) }.
	\end{eqnarray*}
\end{theorem}

\section{\textbf{Simulation study}\label{sec4}}

\subsection{\textbf{Performance and comparative study of} $\widetilde{\Pi }%
	_{n}$ \textbf{and} $\widehat{\Pi }_{n}$}

\noindent In this simulation study we examine the performance of our
estimator $\widetilde{\Pi }_{n}\left( \widehat{\gamma },k\right) $\ given in
(\ref{Br}) and compare it with that of $\widehat{\Pi }_{n}\left( \widehat{%
	\gamma }^{H},k\right) $ given in (\ref{NeMe}). Thus we follow the steps
below.\medskip\ 

\noindent \textbf{Step 1:} We generate $1000$ pseudorandom samples of size $%
n=100,200,500$ and $1000$ from Pareto cdf with $\gamma =0.6.$ \medskip\ 

\noindent \textbf{Step 2:} We estimate the tail index parameter by Hill and
t-Hill estimators $\widehat{\gamma }^{H}(k_{1}^{\ast })$ and $\widehat{%
	\gamma }(k_{2}^{\ast }),$ respectively given in (\ref{Hill}) and (\ref%
{t-Hill}). We adopt the Reiss and Thomas algorithm (see \citeauthor{ReTo07}, %
\citeyear[page 137]{ReTo07}), for choosing the optimal numbers of upper
extremes $k_{1}$ and $k_{2}$. By this methodology, we define the optimal
sample fraction of upper order statistics $k_{j}^{\ast }$ by%
\begin{equation*}
k_{j}^{\ast }:=\arg \min_{k}\frac{1}{k}\sum_{i=1}^{k}i^{\theta }\left\vert 
\widehat{\gamma }_{j}\left( i\right) -\text{median}\left\{ \widehat{\gamma }%
_{j}\left( 1\right) ,...,\widehat{\gamma }_{j}\left( k\right) \right\}
\right\vert ,j=1,2
\end{equation*}%
where $\widehat{\gamma }_{1}=\widehat{\gamma }^{H}$ and $\widehat{\gamma }%
_{2}=\widehat{\gamma }.$ On the light of our simulation study,\textbf{\ }we
obtained reasonable results by choosing $\theta =0.3.$ \medskip

\noindent \textbf{Step 3:} We fix the distortion parameter with respect to
Condition (\ref{condition gamabeta}) by $\rho =1.12,$ then we compute the
bias and RMSE of the four estimators $\widehat{\gamma }^{H}(k_{1}^{\ast }),$ 
$\widehat{\gamma }(k_{2}^{\ast }),$ $\widetilde{\Pi }_{n}\left( \widehat{%
	\gamma },k_{1}^{\ast }\right) $\ and $\widehat{\Pi} _{n}\left( \widehat{%
	\gamma }^{H},k_{2}^{\ast }\right) $. The results are summarized in Table \ref%
{A1}. We see that when dealing with large samples our estimator performs
better.

\begin{table}[h] \centering%
	\begin{tabular}{lcccccccccc}
		& $k_{1}^{\ast }$ & \multicolumn{2}{c}{$\widehat{\gamma }(k_{1}^{\ast })$} & 
		\multicolumn{2}{c}{$\widetilde{\Pi }_{n}\left( \widehat{\gamma },k_{1}^{\ast
			}\right) $} & $k_{2}^{\ast }$ & \multicolumn{2}{c}{$\widehat{\gamma }%
			^{H}(k_{2}^{\ast })$} & \multicolumn{2}{c}{$\widehat{\Pi} _{n}\left( 
			\widehat{\gamma }^{H},k_{2}^{\ast }\right) $} \\ \cline{2-11}
		${\small n}$ &  & bias & RMSE & bias & RMSE &  & bias & RMSE & bias & RMSE
		\\ \hline\hline
		${\small 100}$ & ${\small 10}$ & ${\small -0.0733}$ & ${\small 0.2511}$ & $%
		{\small 0.3618}$ & ${\small 0.5199}$ & ${\small 17}$ & ${\small -0.1641}$ & $%
		{\small 0.2865}$ & ${\small 0.4096}$ & ${\small 0.7332}$ \\ 
		${\small 200}$ & ${\small 23}$ & ${\small -0.0571}$ & ${\small 0.1821}$ & $%
		{\small 0.3562}$ & ${\small 0.5147}$ & ${\small 34}$ & ${\small -0.0993}$ & $%
		{\small 0.2350}$ & ${\small 0.3918}$ & ${\small 0.7185}$ \\ 
		${\small 500}$ & ${\small 62}$ & ${\small -0.0299}$ & ${\small 0.1142}$ & $%
		{\small 0.3404}$ & ${\small 0.4820}$ & ${\small 86}$ & ${\small -0.0301}$ & $%
		{\small 0.0739}$ & ${\small 0.3639}$ & ${\small 0.6936}$ \\ 
		${\small 1000}$ & ${\small 129}$ & ${\small -0.0147}$ & ${\small 0.0798}$ & $%
		{\small 0.1966}$ & ${\small 0.2687}$ & ${\small 169}$ & ${\small -0.0181}$ & 
		${\small 0.0545}$ & ${\small 0.2827}$ & ${\small 0.5279}$ \\ \hline\hline
		&  &  &  &  &  &  &  &  &  & 
	\end{tabular}
	\caption{ $\widehat{\gamma }(k_{1}^{\ast }),$ $\widehat{\gamma }	^{H}(k_{2}^{\ast })$ $\widetilde{\Pi }_{n}\left( \widehat{\gamma},k_{1}^{\ast }\right) $\ and $\widehat{\Pi}
		_{n}\left( \widehat{\gamma},k_{2}^{\ast }\right)$ estimators based on 1000 samples of Pareto-distributed claim amounts
		with tail index 0.6 and distortion parameter $\rho=1.12$. The exact value of the premium is 2.0487.}%
	\label{A1}%
\end{table}%

\subsection{\textbf{Comparative robustness study}}

\noindent In this subsection we study the sensitivity to outliers of $%
\widetilde{\Pi }_{n}\left( \widehat{\gamma },k_{1}^{\ast }\right) $\ and
compare it with that of $\widehat{\Pi}_{n}\left( \widehat{\gamma }%
^{H},k_{2}^{\ast}\right).$ We consider an $\epsilon $-contaminated model
known by mixture of Pareto distributions%
\begin{equation}
F_{\gamma _{1},\gamma _{2},\epsilon }\left( x\right) =1-\left( 1-\epsilon
\right) x^{-1/\gamma _{1}}+\epsilon x^{-1/\gamma _{2}},\text{ }
\label{mixture}
\end{equation}%
where $\gamma _{1},\gamma _{2}>0$ and $0<\epsilon <0.5$ is the fraction of
contamination.\ Note that for $\epsilon =0,$ $\widehat{\gamma }^{H}$ and $%
\widehat{\gamma}$ are asymptotically unbiased. Therefore, for $\epsilon >0,$
the effect of contamination becomes immediately apparent. If $%
\gamma_{1}<\gamma _{2}$ and $\epsilon >0,$ (\ref{mixture}) corresponds to a
Pareto distribution contaminated by a longer tailed distribution. For the
implementation of mixtures models to the outliers study one refers, for
instance, to \cite[page 43]{BaLe95}. In this context, we proceed our study
as follows.\medskip

\noindent First, we consider $\gamma _{1}=0.6,$ $\gamma _{2}=2$ to have the
contaminated model and let $\rho =1.12.$\ Then we consider four
contamination scenarios according to $\epsilon=5\%,$ $10\%,$ $15\%,$ $25\%.$%
\smallskip

\noindent For each value $\epsilon ,$ we generate $1000$ samples of size $%
n=100,$ $200$ and $1000$ from the model (\ref{mixture}). Finally, we compare
the $\widetilde{\Pi}_{n}\left(\widehat{\gamma},k_{1}^{\ast }\right)$ and $%
\widehat{\Pi}_{n}\left(\widehat{\gamma}^{H},k_{2}^{\ast}\right)$ estimators
with this true value, by computing for each estimator, the appropriate bias
and RMSE and summarize the results in Table \ref{A2}.

\begin{table}[h] \centering%
	\begin{tabular}{lccccc}
		&  & \multicolumn{2}{c}{$\widetilde{\Pi }_{n}\left( \widehat{\gamma }%
			,k_{1}^{\ast }\right) $} & \multicolumn{2}{c}{$\widehat{\Pi}_{n}\left( 
			\widehat{\gamma }^{H},k_{2}^{\ast}\right)$} \\ \cline{2-6}
		${\small n}$ & $\%$ contamination & bias & RMSE & bias & RMSE \\ \hline\hline
		${\small 100}$ & ${\small 5}$ & ${\small 0.4043}$ & ${\small 0.6664}$ & $%
		{\small -0.4286}$ & ${\small 1.4727}$ \\ 
		& ${\small 10}$ & ${\small 0.4389}$ & ${\small 0.6862}$ & ${\small -0.7291}$
		& ${\small 1.9123}$ \\ 
		& ${\small 15}$ & ${\small 0.4598}$ & ${\small 0.7464}$ & ${\small -1.2786}$
		& ${\small 2.1247}$ \\ 
		& ${\small 25}$ & ${\small 1.0578}$ & ${\small 1.1305}$ & ${\small -1.2103}$
		& ${\small 2.1828}$ \\ \hline\hline
		$200$ & ${\small 5}$ & ${\small 0.3831}$ & ${\small 0.5532}$ & ${\small %
			-0.4713}$ & ${\small 1.5118}$ \\ 
		& ${\small 10}$ & ${\small 0.3964}$ & ${\small 0.5675}$ & ${\small -1.2496}$
		& ${\small 2.1907}$ \\ 
		& ${\small 15}$ & ${\small 0.4508}$ & ${\small 0.6870}$ & $-1.4355$ & $%
		{\small 2.3107}$ \\ 
		& ${\small 25}$ & ${\small 0.9470}$ & ${\small 1.0197}$ & ${\small -1.7366}$
		& ${\small 2.3274}$ \\ \hline\hline
		${\small 1000}$ & ${\small 5}$ & ${\small 0.2124}$ & ${\small 0.3211}$ & $%
		{\small -0.3794}$ & ${\small 2.1222}$ \\ 
		& ${\small 10}$ & ${\small 0.2329}$ & ${\small 0.3349}$ & ${\small -1.0662}$
		& ${\small 2.3978}$ \\ 
		& ${\small 15}$ & ${\small 0.2931}$ & ${\small 0.3749}$ & ${\small -1.2501}$
		& ${\small 2.0355}$ \\ 
		& ${\small 25}$ & ${\small 0.8124}$ & ${\small 0.9291}$ & ${\small -1.5238}$
		& ${\small 2.3596}$ \\ \hline\hline
		&  &  &  &  & 
	\end{tabular}%
	\caption{$\widetilde{\Pi }_{n}\left( \widehat{\gamma},k_{1}^{\ast }\right)
		$\ and $\widehat{\Pi}_{n}\left( \widehat{\gamma }^{H},k_{2}^{\ast}\right)$
		are based on 1000 samples of mixture of Pareto distributions with tail index
		$0.6,$ $\epsilon=5\%, 10\%, 15\%, 25\%$ and distortion parameter
		$\rho=1.12$. The exact value of the premium is 2.0487.}\label{A2}%
\end{table}%

\noindent As expected, the estimator $\widehat{\Pi }_{n}\left( \widehat{%
	\gamma }^{H},k_{2}^{\ast }\right) $ as well as $\widetilde{\Pi }_{n}\left( 
\widehat{\gamma },k_{1}^{\ast }\right) $ turn out to be more sensitive to
this type of contaminations. For example, in $0\%$ contamination for $n=200$%
, the couple (bias, RMSE) for $\widehat{\Pi }_{n}\left( \widehat{\gamma }%
^{H},k_{2}^{\ast }\right) $ take the values $\left( 0.3918,0.7185\right) $,
while for $15\%$ contamination the bias and the RMSE are given by the couple 
$\left( -1.4355,2.3107\right) $. We may conclude that the bias and RMSE of $%
\widehat{\Pi }_{n}\left( \widehat{\gamma }^{H},k_{2}^{\ast }\right) $
estimator are more sensitive (or note robust) to outliers. However, for $0\%$
contamination the (bias, RMSE) of $\widetilde{\Pi }_{n}\left( \widehat{%
	\gamma },k_{1}^{\ast }\right) $ is $\left( 0.3562,0.5147\right) ,$ while for 
$15\%$ contamination is $\left( 0.4508,0.6870\right) .$ Both the bias and
the RMSE of $\widetilde{\Pi }_{n}\left( \widehat{\gamma },k_{1}^{\ast
}\right) $ estimation are note sensitive to outliers. Then we may conclude
that is the better estimator.

\section{\textbf{Concluding notes\label{sec5}}}

\noindent We showed that the new estimator of premium based on t-Hill
estimator is more robust and performs better than the one based on Hill
estimator proposed by \cite{NeMe09}. Our estimator $\widetilde{\Pi }%
_{n}\left( \widehat{\gamma },k\right) $ is based on Weissman's estimation of
high quantiles, so we would lead to improve our result to use one of several
bias-reduced estimators have been proposed (see for example \citealp{MaBe03}%
).

\section{Proofs\label{sec6}}

\noindent To establish the asymptotic normality of $\widetilde{\Pi }_{\psi
	,n}$\ we need the asymptotic approximation of $\widehat{\gamma }$ with the
same sequence of Brownian bridges as $\widetilde{\Pi }_{\psi ,n},$ for this
reason we give the following results.

\begin{proposition}
	\label{Cor}Assume that the second order condition (\ref{second-cond}) holds
	with $\gamma >1/2$ and let $k=k_{n}$ be an integer sequence satisfying (\ref%
	{K}) and $\sqrt{k}A\left( n/k\right) \rightarrow 0.$ Then,\ there exists a
	sequence of Brownian bridges $\left\{ \mathbb{B}_{n}\left( s\right) ,\text{ }%
	0\leq s\leq 1\right\} $\ such that%
	\begin{equation*}
	\sqrt{k}\left( \widehat{\gamma }-\gamma \right) =\gamma \left( \gamma
	+1\right) ^{2}\int_{0}^{1}s^{\gamma -1}\mathbb{B}_{n}\left( s\right)
	ds+o_{p}\left( 1\right) ,
	\end{equation*}%
	leading to 
	\begin{equation*}
	\sqrt{k}\left( \widehat{\gamma }-\gamma \right) \overset{d}{\rightarrow }%
	\mathcal{N}\left( 0,\frac{\gamma ^{2}\left( 1+\gamma \right) ^{2}}{\left(
		1+2\gamma \right) }\right) ,\text{ as }n\rightarrow \infty ,
	\end{equation*}
\end{proposition}

\begin{proof}
	Our proofs are conducted in the probability space described in Section \ref%
	{sec3}. Then we are entitled to write%
	\begin{equation*}
	S_{k}:=\frac{1}{k}\sum_{j=1}^{k}\frac{X_{n-k:n}}{X_{n-j+1:n}}
	\end{equation*}%
	then $\widehat{\gamma }=S_{k}^{-1}-1$. The asymptotic normality of $\widehat{%
		\gamma }$\ is established in \cite{BeScSt14} by given their Wiener process
	representation. Here we suppose that $\sqrt{k}A\left( n/k\right) \rightarrow
	\lambda =0.$ So 
	\begin{equation*}
	\sqrt{k}\left( S_{k}-\frac{1}{\gamma +1}\right) =\frac{\gamma }{\gamma +1}%
	W_{n}\left( 1\right) -\gamma \int_{0}^{1}t^{\gamma -1}W_{n}\left( t\right)
	dt+o_{p}\left( 1\right) .
	\end{equation*}%
	Note that 
	\begin{equation*}
	\mathbb{B}_{n}\left( s\right) \overset{d}{=}W_{n}\left( s\right)
	-sW_{n}\left( 1\right) ,
	\end{equation*}%
	it follows that 
	\begin{equation*}
	\sqrt{k}\left( S_{k}-\frac{1}{\gamma +1}\right) =-\gamma
	\int_{0}^{1}s^{\gamma -1}\mathbb{B}_{n}\left( s\right) dt+o_{p}\left(
	1\right) .
	\end{equation*}%
	Using the map $g\left( x\right) =1/x-1,$ since $g\left( 1/\left( \gamma
	+1\right) \right) =\gamma $ and applying the delta method yields 
	\begin{equation*}
	\sqrt{k}\left( \widehat{\gamma }-\gamma \right) =\left( \gamma +1\right)
	^{2}\gamma \int_{0}^{1}s^{\gamma -1}\mathbb{B}_{n}\left( s\right)
	ds+o_{p}\left( 1\right) .
	\end{equation*}%
	It is clear that $\sqrt{k}\left( \widehat{\gamma }-\gamma \right) $ is a
	Gaussian rv with mean $0$ and variance $\frac{\gamma ^{2}\left( 1+\gamma
		\right) ^{2}}{\left( 1+2\gamma \right) }.$ This completes the proof of
	Proposition \ref{Cor}.
\end{proof}

\subsection{\textbf{Proof of Theorem \protect\ref{Theorem2}}}

\noindent Making use of Proposition \ref{Cor} and from \cite{NeMeMe07} we
showe that under the assumptions of Theorem \ref{Theorem2}, there exists a
sequence of Brownian bridges $\left\{ \mathbb{B}_{n}\left( s\right) ,0\leq
s\leq 1\right\} $ such that, for all large $n$%
\begin{equation*}
\frac{\left( \Pi _{\psi ,n}^{\left( 1\right) }-\Pi _{\psi }^{\left( 1\right)
	}\right) }{(k/n)^{1/\rho }F^{\leftarrow }\left( 1-k/n\right) }=-\gamma
\left( n/k\right) ^{1/2}\mathbb{B}_{n}\left( 1-k/n\right) +o_{p}\left(
1\right) .
\end{equation*}%
Let $\mathbb{U}$ be the left-continuous inverse of $1/(1-F).$ Note that $%
\mathbb{U}(t)$ is defined for $t>1.$ Let $Y_{1},Y_{2}....$ be independent
and identically distributed rv's with cdf $1-1/y,y>1,$ and let $Y_{1,n}\leq
Y_{2,n}\leq ...\leq Y_{n,n}$ be the associated order statistics. Then, for $%
\rho \geq 1,$ we may rewrite the statistic $\Pi _{\psi ,n}^{\left( 2\right)
} $ as%
\begin{equation*}
\Pi _{\psi ,n}^{\left( 2\right) }=(k/n)^{1/\rho }\frac{\widehat{\gamma }\rho 
}{1-\widehat{\gamma }\rho }X_{n-k:n}.
\end{equation*}
Then%
\begin{equation*}
\Pi _{\psi ,n}^{\left( 2\right) }=(k/n)^{1/\rho }\frac{\widehat{\gamma }\rho 
}{1-\widehat{\gamma }\rho }\mathbb{U}\left( Y_{n-k:n}\right) ,
\end{equation*}%
where $\widehat{\gamma }$ is the t-Hill estimator of $\gamma .$ So we have 
\begin{equation*}
\frac{\sqrt{k}\left( \Pi _{\psi ,n}^{\left( 2\right) }-\Pi _{\psi }^{\left(
		2\right) }\right) }{(k/n)^{1/\rho }\mathbb{U}\left( n/k\right) }%
:=\sum_{j=1}^{4}\Delta _{jn},
\end{equation*}%
where%
\begin{eqnarray*}
	\Delta _{1n}:= &&\sqrt{k}\frac{\widehat{\gamma }\rho }{1-\widehat{\gamma }%
		\rho }\left( \frac{\mathbb{U}\left( Y_{n-k:n}\right) }{\mathbb{U}\left(
		n/k\right) }-\left( \frac{Y_{n-k:n}}{n/k}\right) ^{\gamma }\right) , \\
	\Delta _{2n}:= &&\sqrt{k}\frac{\widehat{\gamma }\rho }{1-\widehat{\gamma }%
		\rho }\left( \left( \frac{Y_{n-k:n}}{n/k}\right) ^{\gamma }-1\right) , \\
	\Delta _{3n}:= &&\sqrt{k}\left( \frac{\widehat{\gamma }\rho }{1-\widehat{%
			\gamma }\rho }-\frac{\gamma \rho }{1-\gamma \rho }\right)
\end{eqnarray*}%
and%
\begin{equation*}
\Delta _{4n}:=\sqrt{k}(k/n)^{-1/\rho }\left( \frac{\frac{(k/n)^{1/\rho }\rho 
	}{\left( 1/\gamma -\rho \right) }\mathbb{U}\left( n/k\right) -\Pi _{\psi
}^{\left( 2\right) }}{\mathbb{U}\left( n/k\right) }\right) .
\end{equation*}%
As showed in \cite{NeMeMe07}, we have : $\Delta _{1n}\rightarrow 0$ and $%
\Delta _{4n}\rightarrow 0$ as $n\rightarrow \infty .$\medskip

\noindent Next, we show that $\Delta _{2n}+\Delta _{3n}$ is asymptotically
normal. Assume, without loss of generality, that the rv's $(Y_{n})_{n\geq 1}$%
are defined on a probability space $(\Omega ,\mathcal{A},\mathbb{P})$ which
carries the sequence $(U_{n})_{n\geq 1}$ in such a way that $Y_{n}=\left(
1-U_{n}\right) ^{-1}$ for $n=1,2,...$ and $Y_{i,n}=\left( 1-U_{i,n}\right)
^{-1},$ $i=1,...,n,$.$\ $Then, this allows us to write $Y_{n-i+1,n}=\left(
1-V_{n}\left( 1-s\right) \right) ^{-1},$ for $\dfrac{i-1}{n}<s\leq \dfrac{i}{%
	n},$ $i=1,...,n$. From \cite{NeMeMe07} we have For $\Delta _{2n}$\ 
\begin{equation*}
\Delta _{2n}=-\left( n/k\right) ^{1/2}\frac{\rho \gamma ^{2}}{1-\gamma \rho }%
\left( 1+o_{p}\left( 1\right) \right) \mathbb{B}_{n}\left( 1-k/n\right) .
\end{equation*}%
For $\Delta _{3n}$ and by using the map $h\left( \theta \right) =\rho
/\left( \frac{1}{\theta }-\rho \right) \ $and applying the delta method
yields: 
\begin{equation*}
\Delta _{3n}=\frac{\rho }{\left( \rho \gamma -1\right) ^{2}}\sqrt{k}\left(
\gamma -\widehat{\gamma }\right) .
\end{equation*}%
From Corollary \ref{Cor} we get%
\begin{equation*}
\Delta _{3n}=\frac{\gamma \rho \left( \gamma +1\right) ^{2}}{\left( \rho
	\gamma -1\right) ^{2}}\int_{0}^{1}s^{\gamma -1}\mathbb{B}_{n}\left( s\right)
ds+o_{p}\left( 1\right) .
\end{equation*}%
Finally we have%
\begin{equation*}
\frac{\sqrt{k}\left( \Pi _{\psi ,n}^{\left( 2\right) }-\Pi _{\psi }^{\left(
		2\right) }\right) }{(k/n)^{1/\rho }F^{\leftarrow }\left( 1-k/n\right) }=%
\frac{\gamma \rho \left( \gamma +1\right) ^{2}}{\left( \rho \gamma -1\right)
	^{2}}\int_{0}^{1}s^{\gamma -1}\mathbb{B}_{n}\left( s\right) ds-\left(
n/k\right) ^{1/2}\frac{\rho \gamma ^{2}}{1-\gamma \rho }\mathbb{B}_{n}\left(
1-k/n\right) +o_{p}\left( 1\right) .
\end{equation*}%
From \cite{NeMeMe07} we have%
\begin{equation*}
\frac{\sqrt{k}\left( \Pi _{\psi ,n}^{\left( 3\right) }-\Pi _{\psi }^{\left(
		3\right) }\right) }{(k/n)^{1/\rho }F^{\leftarrow }\left( 1-k/n\right) }=%
\frac{\int_{k/n}^{1}s^{1/\rho -1}\mathbb{B}_{n}\left( 1-s\right)
	F^{\leftarrow \prime }\left( 1-s\right) ds}{\rho F^{\leftarrow }\left(
	1-k/n\right) \left( k/n\right) ^{1/\rho -1/2}}+o_{p}\left( 1\right) .
\end{equation*}%
Then 
\begin{equation*}
\frac{\sqrt{n}\left( \widetilde{\Pi }_{\psi ,n}-\Pi _{\psi }\right) }{%
	(k/n)^{1/\rho -1/2}F^{\leftarrow }(1-k/n)}=\Lambda \left( \gamma ,\rho
\right) +o_{p}\left( 1\right)
\end{equation*}%
where%
\begin{equation*}
\Lambda \left( \gamma ,\rho \right) :=W_{n1}+W_{n2}+W_{n3}+o_{p}\left(
1\right)
\end{equation*}%
and%
\begin{eqnarray*}
	W_{n1} &:=&-\left( n/k\right) ^{1/2}\gamma \mathbb{B}_{n}\left( 1-k/n\right)
	, \\
	W_{n2} &:=&\frac{\gamma \rho \left( \gamma +1\right) ^{2}}{\left( \rho
		\gamma -1\right) ^{2}}\int_{0}^{1}s^{\gamma -1}\mathbb{B}_{n}\left( s\right)
	ds-\left( n/k\right) ^{1/2}\frac{\rho \gamma ^{2}}{1-\gamma \rho }\mathbb{B}%
	_{n}\left( 1-k/n\right) , \\
	W_{n3} &:=&\frac{\int_{k/n}^{1}s^{1/\rho -1}\mathbb{B}_{n}\left( 1-s\right)
		F^{\leftarrow \prime }\left( 1-s\right) ds}{\rho F^{\leftarrow }\left(
		1-k/n\right) \left( k/n\right) ^{1/\rho -1/2}}.
\end{eqnarray*}%
It is clear that $\Lambda \left( \gamma ,\rho \right) $ is a Gaussian rv
with mean 0 and variance%
\begin{eqnarray*}
	E\left( \Lambda \left( \gamma ,\rho \right) \right) ^{2} &=&E\left(
	W_{n1}^{2}\right) +E\left( W_{n2}^{2}\right) +E\left( W_{n3}^{2}\right)
	+2E\left( W_{n1}W_{n2}\right) \\
	&&+2E\left( W_{n1}W_{n3}\right) +2E\left( W_{n2}W_{n3}\right) .
\end{eqnarray*}%
An elementary calculation gives, we get 
\begin{eqnarray*}
	E\left( W_{n1}^{2}\right) &=&\gamma ^{2}+o\left( 1\right) , \\
	E\left( W_{n2}^{2}\right) &=&\frac{\gamma ^{2}\rho ^{2}}{\left( 1-\gamma
		\rho \right) ^{2}}+\frac{2\gamma }{2\gamma -1}+o\left( 1\right) , \\
	E\left( W_{n3}^{2}\right) &=&\frac{2\gamma ^{2}}{\left( \rho +\gamma \rho
		-1\right) \left( \rho +2\gamma \rho -2\right) }+o\left( 1\right) , \\
	E\left( W_{n1}W_{n2}\right) &=&\frac{\rho \gamma ^{3}}{1-\gamma \rho }%
	+o\left( 1\right) , \\
	E\left( W_{n1}W_{n3}\right) &=&\frac{\gamma \rho }{\rho +\gamma \rho -1}%
	+o\left( 1\right)
\end{eqnarray*}%
and%
\begin{equation*}
E\left( W_{n2}W_{n3}\right) =-\frac{\gamma ^{3}\rho ^{2}}{\left( \gamma \rho
	-1\right) \left( \rho +\gamma \rho -1\right) }+o\left( 1\right) .
\end{equation*}%
The proof of Theorem \ref{Theorem2} is completed by combining all the
preceding results.\hfill $\square $

\end{document}